\documentclass[a4paper]{amsart}

\usepackage{amssymb}
\usepackage{amsthm}
\usepackage{amsmath}
\usepackage{amscd}
\usepackage{float}
\usepackage[utf8]{inputenc}
\usepackage[T1]{fontenc}
\usepackage{lmodern}
\usepackage[parfill]{parskip}
\usepackage{url}
\usepackage{hyperref}
\usepackage{graphicx}
\hypersetup{linktocpage}
\usepackage[all]{xy}
\usepackage{dsfont}

\newcommand\Id{{\mathds{1}}}
\newcommand\op[1]{\mathop{\rm #1}\nolimits}

\newcommand{\om}{\omega}
\newcommand{\ka}{\kappa}

\newcommand{\ba}{\mathcal{G}}

\newcommand{\fg}{\mathfrak g}
\newcommand{\fp}{\mathfrak p}

\newcommand{\fh}{\mathfrak h}

\newcommand{\na}{\nabla}
\newcommand{\U}{\Upsilon}

\newcommand{\R}{\mathbb{R}}

\newcommand{\ad}{{\rm ad}}

\newcommand{\C}{\mathbb{C}}

\newcommand{\V}{\mathbb{V}}
\newcommand{\p}{\mathfrak{p}}

\newcommand{\End}{\text{End}}

\renewcommand{\H}{\mathbb{H}}
\renewcommand\sp{\mathfrak{sp}}
\renewcommand\sl{\mathfrak{sl}}

\theoremstyle{plain}

\theoremstyle{plain}
\newtheorem{thm*}{Theorem}[section]
\newtheorem{prop*}[thm*]{Proposition}
\newtheorem{lem*}[thm*]{Lemma}
\newtheorem{cor*}[thm*]{Corollary}

\theoremstyle{definition}
\newtheorem{def*}{Definition}

\theoremstyle{remark}
\newtheorem{rem*}{Remark}

\def\sli{\ar@{-}}
\newlength{\lengtharrow}      
\newlength{\lengtharrowhalf}  
\newlength{\lengtharrowS}     
\newlength{\lengtharrowShalf} 
\newlength{\spaceunderA}      
\newlength{\spaceundercoef}   
\newlength{\spaceundercoefS}  

\newcommand{\Diagram}[7][0pt]{
  \addtolength{\spaceunderA}{#1}
  \parbox[c][\spaceunderA][t]{1in}{\xymatrix@=\lengtharrowS@R=0pt{
                                                                           \\
     *\txt{$\scriptstyle \rule[\spaceundercoef]{0pt}{0mm} #2$}
    &*\txt{$\scriptstyle \rule[\spaceundercoef]{0pt}{0mm} #3$}
    &*\txt{$\scriptstyle \rule[\spaceundercoef]{0pt}{0mm} #4$}
    &*\txt{$\scriptstyle \rule[\spaceundercoef]{0pt}{0mm} #5$}
    &&*\txt{$\scriptstyle \rule[\spaceundercoef]{0pt}{0mm} #6$}
    &*\txt{$\scriptstyle \rule[\spaceundercoef]{0pt}{0mm} #7$}             \\
     *\txt{$\bullet$}       \sli[r]
    &*\txt{$\times$}       \sli[r]
    &*\txt{$\bullet$}       \sli[r]
    &*\txt{$\circ$}       \sli[r]
    &*\txt{$\, \cdots$} \sli[r]
    &*\txt{$\circ$}       \sli[r]
    &*\txt{$\bullet$}
  }}
}
\newcommand{\DiagramReduced}[9][0pt]{
  \addtolength{\spaceunderA}{#1}
  \parbox[c][\spaceunderA][t]{1in}{\xymatrix@=\lengtharrowS@R=0pt{
                                                                           \\
     *\txt{$\scriptstyle \rule[\spaceundercoef]{0pt}{0mm} #2$}
    &*\txt{$\scriptstyle \rule[\spaceundercoef]{0pt}{0mm} #3$}
    &*\txt{$\scriptstyle \rule[\spaceundercoef]{0pt}{0mm} #4$}
    &*\txt{$\scriptstyle \rule[\spaceundercoef]{0pt}{0mm} #5$}
    &*\txt{$\scriptstyle \rule[\spaceundercoef]{0pt}{0mm} #6$}
    &&*\txt{$\scriptstyle \rule[\spaceundercoef]{0pt}{0mm} #7$}
    &*\txt{$\scriptstyle \rule[\spaceundercoef]{0pt}{0mm} #8$}
    &*\txt{$\scriptstyle \rule[\spaceundercoef]{0pt}{0mm} #9$}            \\
     *\txt{$\bullet$}       \sli[r]
    &*\txt{$\times$}       \sli[r]
    &*\txt{$\bullet$}       \sli[r]
    &*\txt{$\times$}       \sli[r]
    &*\txt{$\bullet$}       \sli[r]
    &*\txt{$\, \cdots$}    \sli[r]
    &*\txt{$\bullet$}       \sli[r]
    &*\txt{$\times$}       \sli[r]
    &*\txt{$\bullet$}
  }}
}
\newcommand{\DiagramDisconnected}[9][0pt]{
  \addtolength{\spaceunderA}{#1}
  \parbox[c][\spaceunderA][t]{1in}{\xymatrix@=\lengtharrowS@R=0pt{
                                                                           \\
     *\txt{$\scriptstyle \rule[\spaceundercoef]{0pt}{0mm} #2$}
    &*\txt{$\scriptstyle \rule[\spaceundercoef]{0pt}{0mm} #3$}
    &*\txt{$\scriptstyle \rule[\spaceundercoef]{0pt}{0mm} #4$}
    &*\txt{$\scriptstyle \rule[\spaceundercoef]{0pt}{0mm} #5$}
    &*\txt{$\scriptstyle \rule[\spaceundercoef]{0pt}{0mm} #6$}
    &&*\txt{$\scriptstyle \rule[\spaceundercoef]{0pt}{0mm} #7$}
    &*\txt{$\scriptstyle \rule[\spaceundercoef]{0pt}{0mm} #8$}
    &*\txt{$\scriptstyle \rule[\spaceundercoef]{0pt}{0mm} $}
    &*\txt{$\scriptstyle \rule[\spaceundercoef]{0pt}{0mm} #9$}            \\
     *\txt{$\bullet$}       \sli[]
    &*\txt{\ \ \ \ }       \sli[]
    &*\txt{$\bullet$}       \sli[]
    &*\txt{\ \ \ \ }       \sli[]
    &*\txt{$\bullet$}       \sli[]
    &*\txt{\ \ \ \ }    \sli[]
    &*\txt{$\bullet$}       \sli[r]
    &*\txt{$\circ$}       \sli[r]
    &*\txt{$\, \cdots$}     \sli[r]
    &*\txt{$\bullet$}
  }}
}
\newcommand{\DiagramNoX}{
\xymatrix{
     *\txt{$\bullet$}       \sli[]
    &*\txt{\ \ \ \ }       \sli[]
    &*\txt{$\bullet$}       \sli[r]
    &*\txt{$\circ$}       \sli[r]
    &*\txt{$\bullet$}       \sli[r]
    &*\txt{$\, \cdots$} \sli[r]
    &*\txt{$\circ$}       \sli[r]
    &*\txt{$\bullet$}
}}

\newcommand{\DiagramNoCoef}{
\xymatrix{
     *\txt{$\bullet$}       \sli[r]
    &*\txt{$\times$}       \sli[r]
    &*\txt{$\bullet$}       \sli[r]
    &*\txt{$\circ$}       \sli[r]
    &*\txt{$\bullet$}       \sli[r]
    &*\txt{$\, \cdots$} \sli[r]
    &*\txt{$\circ$}       \sli[r]
    &*\txt{$\bullet$}
}}

\begin{document}
\title[Submaximally Symmetric Almost Quaternionic Structures]{Submaximally Symmetric\\
Almost Quaternionic Structures}
\author[Boris Kruglikov, Henrik Winther and Lenka Zalabov\'{a}]{Boris Kruglikov$^\dag$, Henrik Winther$^\dag$ and Lenka Zalabov\'{a}$^\ddag$}
\address{$\dag$ Institute of Mathematics and Statistics, University of Troms\o, Troms\o\ 90-37, Norway.
E-mails: \textsc{boris.kruglikov@uit.no}, \textsc{henrik.winther@uit.no}.}
\address{$\ddag$ Institute of Mathematics and Biomathematics, Faculty of Science,
 University of South Bohemia in \v{C}esk\'{e} Bud\v{e}jovice, Brani\v{s}ovsk\'{a} 1760, \v Cesk\'{e} Bud\v{e}jovice, 370 05, Czech Republic.
E-mail: \textsc{lzalabova@gmail.com}.}
 \thanks{}

\maketitle
\begin{abstract}
The symmetry dimension of a geometric structure is the dimension of its symmetry algebra. We investigate
symmetries of almost quaternionic structures of quaternionic dimension $n$. The maximal possible
symmetry is realized by the quaternionic projective space $\H P^n$, which is flat and has the
symmetry algebra $\frak{sl}(n+1,\H)$ of dimension $4n^2+8n+3$. For non-flat almost quaternionic manifolds we
compute the next biggest (submaximal) symmetry dimension.
We show that it is equal to $4n^2-4n+9$ for $n>1$ (it is equal to 8 for $n=1$).
This is realized both by a quaternionic structure (torsion--free) and by an almost quaternionic structure with vanishing quaternionic Weyl curvature.
\end{abstract}


 \section{Introduction}
An \emph{almost quaternionic structure} on a manifold $M$ is a rank three subbundle $Q \subset \End(TM)$ such that locally (in a neighbourhood of each point) we can find a basis $I,J,K$ of $Q$ with $I^2=J^2=K^2=-\Id$ and $IJ=K$.
A manifold $M$ with a fixed almost quaternionic structure $Q$ is called an \emph{almost quaternionic manifold}. A (local) automorphism of $(M,Q)$ is a (local) diffeomorphism of $M$ that preserves $Q$. There exists a class of the so--called \emph{Oproiu connections} $[\na^{Op}]$ on $(M,Q)$ that preserve $Q$ and share the same minimal torsion $T_\nabla$, which equals to the structure torsion of $Q$ \cite{Alekseevsky}. If $\na^{Op}$ is torsion--free, then $(M,Q)$ is a \emph{quaternionic manifold}.

An almost quaternionic manifold $(M,Q)$ can be equivalently described as a normal parabolic geometry $(\ba \to M, \om)$ of type $PGL(n+1,\H)/P$, where $P$ is the stabilizer of a quaternionic line in $\H^{n+1}$ \cite{parabook}. The fundamental invariant of each parabolic geometry is its \emph{harmonic curvature} $\ka_H$, which has two components in the almost quaternionic case: the \emph{torsion} $\ka_1$ of homogeneity $1$ and the \emph{quaternionic Weyl curvature} $\ka_2$ of homogeneity $2$. In particular, $\ka_1$ coincides with the torsion $T_\nabla$ of arbitrary $\na^{Op}$ and vanishes for quaternionic geometries.

The quaternionic projective space $\H P^n$ is the set of quaternionic lines in $\H^{n+1}$, and the group $PGL(n+1,\H)$ acts transitively on $\H P^n$ as automorphisms of the natural quaternionic structure. We can view $P$ as (the quotient of) the stabilizer of the first basis vector of $\H^{n+1}$.  Then, $\H P^n=PGL(n+1,\H)/P$ is the \emph{flat model} of (almost) quaternionic geometry. The flat model has vanishing harmonic curvature and conversely, each (almost) quaternionic geometry such that $\ka_H \equiv 0$ is locally equivalent to the flat model. In particular, every local automorphism of $\H P^n$ uniquely extends to a global one, and it is exactly the left multiplication by an element of $PGL(n+1,\H)$. The space $\H P^n$ has maximal possible dimension of the symmetry algebra among all (almost) quaternionic manifolds with fixed quaternionic dimension $n$, that is $\dim\frak{sl}(n+1,\H)=4(n+1)^2-1$ for $\dim M=4n$.

For curved almost quaternionic structures, local automorphisms generally do not extend to global ones.  Therefore we consider infinitesimal symmetries, which correspond to local automorphisms. We focus on the problem of establishing
the \emph{submaximal symmetry dimension}, i.e.\ the maximal dimension of the symmetry algebra of  an almost
quaternionic structure with $\ka_H \not\equiv 0$ and fixed quaternionic dimension. Specifically, we answer the following question:
 \begin{center}
{\it When an almost quaternionic manifold $(M^{4n},Q)$ is not everywhere flat, what \\
is the maximal dimension $\mathfrak{S}$ of its Lie algebra of infinitesimal symmetries?}
 \end{center}

 \begin{rem*}
The submaximal dimension of the automorphism groups (without the requirement $\ka_H \not\equiv 0$) is
$\dim P=4n^2+4n+3$. This is achieved on the flat manifold $M=\H P^n\setminus\{p\}$ for some $p \in \H P^n$.
However the symmetry algebra of this $(M,Q)$ is of maximal dimension $4n^2+8n+3$.
 \end{rem*}

From the point of view of parabolic geometry, a model with the symmetry algebra of submaximal dimension
typically has exactly one non-zero component of its harmonic curvature \cite{KT}. Sometimes, the same submaximal bound is achieved for different non-zero components of $\ka_H$. We will show that this is the case with almost quaternionic structures. Our main result is the following.

\begin{thm*}\label{mainthm}
The maximal dimension of the symmetry algebra of almost quaternionic structures $(M,Q)$ with $\dim M=4n>4$ and
$\ka_H=(\kappa_1,\kappa_2)\not\equiv 0  $ is
 $$
\mathfrak{S}=4n^2-4n+9.
 $$
This is realized in both cases, when $\kappa_1\equiv0$ and when $\kappa_2\equiv0$.
\end{thm*}

We exclude the case $n=1$ due to the exceptional isomorphism $\frak{sl}(2,\H) \simeq \frak{so}(1,5)$. In this case the geometry $PGL(2,\H)/P$ can be interpreted as a four-dimensional Riemannian conformal geometry and $\ka_H$ has two components of homogeneity $2$, which are the self--dual and anti--self--dual parts of the Weyl curvature. The submaximal symmetry dimension is $8$ and is achieved by $M=\C P^2$ \cite{E2,KT}.

\smallskip
\noindent {\bf Acknowledgements:}
The authors thank Jan Gregorovi\v{c} for suggesting that the submaximal symmetry dimensions for $n=2$ should be 17.
Lenka Zalabov\'{a} thanks Norway Grants NF-CZ07-INP-4-2382015 and NF-CZ07-INP-5-4362015 for financial support and the University of Troms\o\ for hospitality.

\section{Background on almost quaternionic and related geometries}

Almost quaternionic geometries are closely related to the projective and $c$--projective geometries, so we recall basic concepts common to these. Two (real) connections on a manifold $M$ of dimension $n$ are
\emph{projectively equivalent} if their unparameterized geodesics, i.e.\ curves satisfying $\na_{\dot{\gamma}} \dot{\gamma} \in \langle \dot{\gamma} \rangle$, coincide. Here $\langle - \rangle$ denotes the linear span over $C^\infty(M)$. Projectively equivalent connections do not necessarily have the same torsion, and each connection $\na$ is projectively equivalent to a torsion--free connection $\na-{1 \over 2}T_\na$. Two connections $\na$ and $\hat \na$ with the same torsion are projectively equivalent if and only if there is a one-form $\U \in \Omega^1(M)$ such that
 $$
\na - \hat \na= \Id\otimes(\U\circ\Id)+(\U\circ\Id)\otimes\Id.
 $$
A fixed class of torsion--free projectively equivalent connections $[\na]$ on a manifold $M$ is a \emph{projective structure} on $M$. It is proven in  \cite{E1} that submaximal symmetry dimension in the class of projective structures of dimension $n > 2$ is equal to $(n-1)^2+4$ (for $n=2$ it is 3), see also \cite{KT}.

A generalization of this concept to almost complex manifolds leads to $c$--projective structures. A connection $\na$ on $M$ of dimension $2n>2$ with almost complex structure $J$ is called \emph{complex} if $\na J=0$. Each almost complex manifold $(M,J)$ admits complex connections, because for arbitrary $\na$ the connection ${1 \over 2}(\na-J \na J)$ is complex. A complex connection $\na$ can be chosen \emph{minimal} meaning $T_\na={1\over 4}N_J$, where
 $$
N_J(X,Y) =[JX,JY]-J[JX,Y]-J[X,JY]-[X,Y]
 $$
is the Nijenhuis tensor.

A curve $\gamma$ on $M$ is \emph{J-planar} if
$\na_{\dot{\gamma}} \dot{\gamma} \in \langle \dot{\gamma},J\dot{\gamma} \rangle$ for a complex connection $\na$. Two complex connections on $(M,J)$ are \emph{$c$--projectively equivalent} if they share the same $J$--planar curves. Two complex connections $\na$ and $\hat \na$ with the same torsion are $c$--projectively equivalent if and only if there is a one-form $\U \in \Omega^1(M)$ such that
 $$
\na - \hat\na= \sum_{A \in\{\Id,J\}}A^2\bigl(A\otimes(\U \circ A)+(\U \circ A) \otimes A\bigr).
 $$
An \emph{almost $c$--projective structure} on $(M,J)$ is a class of $c$--projectively equivalent complex connections $[\na]$ sharing the same fixed torsion. It is proven in \cite{KMT} that the submaximal dimension in the class of almost $c$--projective structures of (complex) dimension $n$ is equal to $2n^2-2n+4$ for $n\neq3$ and $18$ for $n=3$.

Let us return to almost quaternionic structures. Consider an almost quaternionic manifold $(M,Q)$ of dimension $4n$.
Analogously to the almost complex case, this admits a quaternionic connection.
Indeed, for any local basis $b=(I,J,K)$ of $Q$ and a linear connection $\na$, the connection
$\na_b:={1\over 4}(\na -I\na I -J \na J - K \na K)$ is quaternionic. Any other choice
$\hat b=(\hat I, \hat J, \hat K)$ is related to $b$ via a transformation from $SO(Q)$, so
$\na_{\hat b}={1\over 4}(\na -\hat I\na \hat I -\hat J \na \hat J - \hat K \na \hat K)$ coincides with $\na_b$.
Denote $B:={1 \over 6}(N_I+N_J+N_K)$. The canonical {\em structure tensor} of $Q$ is given by
 $$
T_Q:=B+\delta(\tau_I \otimes I)+\delta(\tau_J \otimes J)+\delta(\tau_K \otimes K),
 $$
where $\tau_A(X)={1 \over4n-2} Tr(AB(X))$ for $A=I,J,K$ and $\delta:T^*M \otimes T^*M \otimes TM \to \wedge^2 T^*M \otimes TM$ denotes the Spencer operator of alternation \cite{Alekseevsky}.
A quaternionic connection can be chosen minimal meaning its torsion coincides with $T_Q$.
An almost quaternionic structure $Q$ is a \emph{quaternionic structure} if $T_Q$ vanishes.

A curve $\gamma$ is called \emph{$Q$--planar} if $\na_{\dot{\gamma}} \dot{\gamma} \in \langle \dot{\gamma},I\dot{\gamma},J\dot{\gamma},K\dot{\gamma}  \rangle$ for a quaternionic connection $\na$. Two quaternionic connections $\nabla$ and $\hat\nabla$ on $(M,Q)$ with the same torsion share the same $Q$--planar curves if and only if there is a one-form $\U \in \Omega^1(M)$ such that
 $$
\nabla-\hat\nabla = \sum_{A\in\{\Id,I,J,K\}} A^2\Bigl(A\otimes (\U \circ A)+(\U\circ A)\otimes A\Bigr).
 $$
Analogously to the $c$--projective case, we fix the class of connections $[\na]$ sharing the same $Q$--planar curves and with the minimal torsion $T_\na=T_Q$. These are called \emph{Oproiu connections}. The $Q$--planar curves are the (unparameterized) geodesics of all Oproiu connections \cite{HS}. Given an arbitrary quaternionic connection, one can construct an Oproiu connection by an explicit formula \cite[\S3.11]{Alekseevsky}.

An almost quaternionic structure is quaternionic if and only if some (and thus any) Oproiu connection $\na$ is torsion--free. In that case, the curvature $R_\na$ of an Oproiu connection $\na$ decomposes as $R_\na=W_\na+P_\na$ with respect to the structure group $G_0=Sp(1)GL(n,\H)$, where the trace-free part $W_\na$ is the
\emph{Weyl tensor} and $P_\na$ is the \emph{Ricci tensor} of $\na$, see \cite{Alekseevsky}.
It turns out that the Weyl part $W_\na$ of $R_\na$ does not depend on the choice of Oproiu connection and is a complete obstruction to the flatness of a quaternionic structure.

\begin{rem*}
All three geometries discussed in this section can be described as parabolic geometries of type
$PGL(n+1,\mathbb{K})/P$, where $\mathbb{K}=\R, \C, \H$ and $P$ is the stabilizer of a $\mathbb{K}$-line
in $\mathbb{K}^{n+1}$, see \cite{parabook}. This explains many similarities between them.
\end{rem*}

\section{Setup from parabolic geometries and annihilators}\label{Ssetup}
In this section, we summarize basic facts about almost quaternionic structures from the parabolic viewpoint.
Consider the Lie algebra $\fg=\frak{sl}(n+1,\H)$, which is a real form of $A_{2n+1}=\frak{sl}(2n+2,\C)$.
The parabolic subalgebra $\fp=\op{Lie}(P)$, determining $|1|$-grading $\fg=\fg_{-1} \oplus \fg_0 \oplus \fg_1$,
is encoded by the following Satake diagram:
 $$
\DiagramNoCoef
 $$
This grading can be viewed via the matrix $(1,n)$-block decomposition $\begin{pmatrix}a & p \\ v & A\end{pmatrix}$
in the standard representation on $\H^{n+1}=\H\times\H^n$: $\fg_{-1}=\{v\in\H^n\}$, $\fg_{1}=\{p\in\H^{*n}\}$ and
$\fg_0=\{(a,A)\in\H\oplus\frak{gl}(n,\H):\op{Re}(a)+\op{Re}(\op{tr}A)=0\}$.
In particular, the real part of $a\in\H$ is determined by $\op{tr}A$ and the imaginary part belongs to $\frak{sp}(1)$. Thus the reductive Lie algebra $\fg_0$ can be equivalently viewed as $\frak{sp}(1)+\frak{gl}(n,\H)$, and this
further decomposes as $\fg_0=\frak{sp}(1)+\R Z+\frak{sl}(n,\H)$, where the semisimple part is $\fg_0^{ss}=\frak{sp}(1)+\frak{sl}(n,\H)$ and the grading element
$Z=\op{diag}\bigr(\frac{n}{n+1},\frac{-1}{n+1},\dots,\frac{-1}{n+1}\bigl)$
generates the center $\frak{z}(\fg_0)$. The Lie algebra $\fg_0^{ss}$ is encoded by the Satake diagram
produced by removing the crossed node and adjacent edges:
 $$
\DiagramNoX
 $$

A fundamental invariant of a regular normal parabolic geometry is the \emph{harmonic curvature} $\ka_H$, taking
values in the $G_0$--module $H^2(\fg_{-},\fg)$ (that is the Lie algebra cohomology of $\fg_-$ with values in $\fg$;
in the quaternionic case the regularity requirement is vacuous, i.e.\ $H^2=H^2_+$ has
positive homogeneity because the geometry is $|1|$-graded). This is a completely reducible module,
and its two irreducible components $H^2_1$ and $H^2_2$ (the subscript denotes homogeneity of the cohomology
with respect to $Z$) yield the corresponding decomposition of $\ka_H$ into two summands:
 \begin{itemize}
\item the \emph{torsion} $\ka_1$ of homogeneity $1$ valued in $H_1^2(\fg_{-1},\fg)$, and
\item the \emph{quaternionic Weyl curvature} $\ka_2$ of homogeneity $2$ valued in $H_2^2(\fg_{-1},\fg)$.
 \end{itemize}
The harmonic curvature $\kappa_1$ coincides with the torsion of an arbitrary Oproiu connection, and if the torsion vanishes, then the harmonic curvature $\kappa_2$ coincides with Weyl tensor of an arbitrary Oproiu connection. For an almost quaternionic structure that is not quaternionic $\ka_H=\ka_1+\ka_2$ and both components are non-vanishing in general.

To compute the structure of these modules,  where $\ka_1$ and $\ka_2$ have their values,
we invoke the complexification:
the corresponding parabolic subalgebra $\fp_\C \subset \fg_\C$ induces $|1|$-grading of $\fg_\C$ and
$H^2(\fg_{-} ,\fg) \otimes \C \simeq H^2 ((\fg_\C)_-,\fg_\C)$.
Explicit algorithmic description of the $G^\C_0$--module structure of the latter follows from the Kostant's version
of the Bott-Borel-Weil theorem \cite{parabook}.

In the case of almost quaternionic structures, the submodules corresponding to the torsion $\ka_1$ and quaternionic Weyl curvature $\ka_2$ are encoded by minus lowest weights (adapting the convention of \cite{BE}) as follows,
where the number over the $i$'th node is the coefficient of the corresponding fundamental weight $\omega_i$:
 $$
\ka_1:\ \ \ \Diagram{3}{-3}{0}{1}{0}{1} \ \ \quad \ \ \ka_2:\ \ \ \Diagram{0}{-4}{3}{0}{0}{1}
 $$

 \begin{rem*}
Let us point out that $H^2_+(\fg_{-},\fg)$ is a real $\fg_0$--module that we identify with a real
$\fg_0$--submodule of $H^2_+(\fg_{-} ,\fg) \otimes \C$. Note also that minus the lowest weight
is equal to the highest weight of the dual module.
 \end{rem*}

Let us recall how to get a universal upper bound $\mathfrak{U}$ on the submaximal symmetry dimension $\mathfrak{S}$, and explain the role of the $G_0$-module $H^2(\fg_{-},\fg)$. Each infinitesimal symmetry has to preserve (each component of) $\ka_H$, and thus it belongs to the annihilator of $\ka_H$ in $\fg$. On the level of $\fg_0$ this implies that the grading $\mathfrak{s}=\oplus_i\mathfrak{s}_i$ associated to the filtration by stabilizer on the
symmetry algebra (determined by a point $u\in\mathcal{P}$ of the Cartan bundle \cite{KT}) satisfies
 $$
\mathfrak{s}_0\subset\mathfrak{a}_0=\{  \phi \in \fg_0: \phi\cdot\ka_H=0  \}.
 $$
Furthermore it is proven in \cite{KT} that $\mathfrak{S} \leq \mathfrak{U}$ for
 $$
\mathfrak{U}=\max\{\dim (\mathfrak{a}^\psi):0\neq\psi \in H^2_+(\fg_{-},\fg)\},
 $$
where $\mathfrak{a}^\psi$ is the \emph{Tanaka prolongation} of the pair $(\fg_-, \frak{a}_0^\psi)$, and $\frak{a}_0^\psi$ is the annihilator of $\psi$ in $\fg_0$. Moreover, $\mathfrak{S} \leq \mathfrak{U} \leq \mathfrak{U}^\C$, where $\mathfrak{U}^\C$ is the universal upper bound for the complexified geometry, and the
universal upper bound is realized by the stabilizer of minus the lowest weight vector in the complex case.

By \cite[Corollary 3.4.8]{KT} the parabolic structures of type $A_{2n+1}/P_2$ are \emph{prolongation--rigid},
i.e.\ the Tanaka prolongation $\mathfrak{a}_+^\psi=0$ for any $\psi\neq0$. This implies the corresponding
statement for real geometries, see \cite[Proposition 3]{K}. Thus almost quaternionic structures are prolongation--rigid $\mathfrak{a}^\psi_1=0$, and so $\mathfrak{a}^\psi=\fg_{-1}\oplus \mathfrak{a}^\psi_0$ for each non-zero element $\psi \in H^2_+(\fg_{-1},\fg)$.

Let $\mathfrak{S}_i$ be the maximal symmetry dimension of the almost quaternionic geometry
of $\dim_\H=n$ in the case $\kappa_i\not\equiv0$.
We are going to bound this
$\mathfrak{S}_i\leq\mathfrak{U}_i=\max\{\dim(\mathfrak{a}^\psi):0\neq\psi\in H^2_i(\fg_{-1},\fg)\}$
and prove that the submaximal symmetry dimension is
 $$
\mathfrak{U}_1=\mathfrak{S}_1=\mathfrak{S}=\mathfrak{S}_2=\mathfrak{U}_2.
 $$
We, however, cannot directly apply the methods from complex parabolic geometry. It turns out that
the corresponding upper bounds are strictly less than the upper bounds for the complexification: $\mathfrak{U}_i<\mathfrak{U}_i^\C$ and thus $\mathfrak{U}<\mathfrak{U}^\C$.

A similar phenomenon was noticed for Lorentzian conformal geometries in \cite{DT}, where the submaximal symmetry dimension was computed by listing all subalgebras of high dimensions that stabilize a non-zero element in the harmonic curvature module. In this paper, we choose a different approach by identifying a real analogue to the lowest weight vector in the real harmonic curvature module.

\section{Minimal orbits}
Recall that in the case of complex parabolic geometries, obtaining the upper symmetry bound is based on the Borel fixed point theorem, which states that there is a unique closed orbit, which is of minimal dimension, in the projectivization of $H^2 (\fg_-,\fg)$. Then the upper bound is given by dimension of the stabilizer of a weight vector corresponding to minus the lowest weight (generating the minimal orbit). The Borel theorem cannot be applied in the case of almost quaternionic structures, but we still consider the $\fg_0$--orbits in the projectivization of $H^2 (\fg_{-1},\fg)$ to find one of the minimal dimension. The following statement is immediate.
 \begin{lem*}
The annihilator of $0\neq\ka_i\in H^2_i$ ($i=1,2$) is of maximal dimension in $\fg_0$ if and only if the
$G_0$--orbit through $\ka_i$ has minimal dimension in the projectivization of $H^2(\fg_{-1},\fg)$.
 \end{lem*}

We will need the following result on existence of closed orbits. Here we denote by $[v]$ the projection of
a non-zero vector $v\in\V$ to the projective space $P\V$.

 \begin{lem*}\label{minorbitisclosed}
Let $\V$ be a real irreducible $L$--module for a real Lie group $L$. Then there exists $0\neq v\in\V$
such that for $[v]\in P\V$ the orbit $L\cdot [v]\subset P\V$ is closed and of minimal dimension.
 \end{lem*}

 \begin{proof}
We claim that any orbit $L\cdot [v]$ of minimal dimension $d=\dim L\cdot [v]$ is closed.
Indeed, consider complexification of the group, the action and the representation.
The element $v\in\V+0\cdot i\subset \V^\C$ determines the complex orbit $L^\C\cdot[v]\subset P\V^\C$
of the same {\it complex} dimension $d$ (because the annihilator of $v$ in $\mathfrak{l}^\C$ intersects
$\mathfrak{l}=\op{Lie}(L)$ by the annihilator of $v$ in $\mathfrak{l}$).
If closure of the orbit $L\cdot [v]$ contains another orbit
$L\cdot[v']$ (necessary of the same dimension $d$), then closure of the complex orbit $L^\C\cdot[v]$
contains the complex orbit $L^\C\cdot[v']$ (again of the same complex dimension $d$). But this cannot
happen in the complex case: closure of any orbit can contain only orbits of strictly smaller dimensions
(this follows from decomposition of a vector from $\V^\C$ via weight vectors with respect to the standard
Borel subalgebra: every orbit has a representative with the coefficient 1 at the highest weight vector).
 \end{proof}

Let's consider the minimal orbits in the projectivizations $PH^2_1$ and $PH^2_2$ of
both irreducible components of $H^2(\fg_{-1},\fg)$.

We will need the notations describing parabolic subalgebras of real semi-simple Lie algebras. The conjugacy
classes of such are in bijection with some subsets of the Satake diagram corresponding to a fixed choice of (maximally non-compact) Cartan subalgebra. These will be denoted by crossing out certain white nodes on
the Satake diagram, cf.\ \cite{parabook}.
Let $\Sigma$ be the set of crossed out nodes. We denote by $\p_\Sigma^{}$ the standard parabolic subalgebra corresponding to $\Sigma$. This is the subalgebra generated by all positive restricted root vectors, and negative restricted root vectors for which the restricted roots are linear combinations of the complement to $\Sigma$. The semi-simple Levi factor $\p_\Sigma^{ss}$ is given by the Satake diagram with $\Sigma$ removed.

Parabolic subalgebras are in bijection with $\mathbb{Z}$-gradings of semi-simple Lie algebras
$\fg=\fg_{-k}\oplus\ldots\oplus\fg_k$. Indeed, the parabolic subalgebra contains a unique grading element
$Z\in\fg$, for which $\mathfrak{ad}_Z|_{\fg_n}=n\,\text{Id}_{\fg_n}$.
Conversely, given a grading of $\fg$, the subalgebra $\fp=\fg_{\ge0}$ of non-negative gradation is parabolic
(all parabolics arise in this way \cite{parabook}).

Recall that $\fg_0=\sp(1)+\R Z+\mathfrak{sl}(n,\H)$. Let $\V$ be a $\fg_0$-module which is irreducible under the restricted representation of $\sl(n,\H)$. The group $\tilde G_0=PGL(n,\H)$ acts effectively on $P\V$, and the Lie algebra $\tilde \fg_0=\mathfrak{sl}(n,\H)$ is simple.
Consider the following parabolic subgroups of $\tilde G_0$:
$H=\tilde{P}_2$ in the case $n=2$, and $H=\tilde{P}_{2,2n-2}$ in the case $n>2$
(tilde in $P$ indicates that numeration of the parabolics is with respect to $\tilde G_0$).
This parabolic determines the grading on $\tilde \fg_0$ with respect to which the parabolic $(\tilde \fg_0)_{\ge0}$
is equal to $\fh=\fh_0\oplus \fh_+$ as a vector space, where
 \begin{align*}
&\fh_0=\sp(1)\oplus\frak{gl}(n-2,\H)\oplus\sp(1)\oplus\R Z',\\
&\fh_+=\fh_1\oplus\fh_2=\mathfrak{heis}(8n-12,\H),
 \end{align*}
and $\mathfrak{heis}(8n-12,\H)$ is the quaternionification of the real (nilpotent)
Heisenberg algebra $\mathfrak{heis}(2n-3)$.
The action of $\mathfrak{heis}(8n-12,\H)$ on $\H^n$ is given by $n\times n$ quaternionic matrices with zeroes
everywhere except for the first row and the last column, and with zeroes on the diagonal.

To distinguish the summand $\mathfrak{sp}(1)$ in $\fg_0$ from these in $\fh_0$ we will use the
notations $\mathfrak{sp}(1)_{\rm left}$ and $\mathfrak{sp}(1)_{\rm right}$ for the latter
(marking them in the appearing order).

The grading element $Z' \in \tilde \fg_0$ acts on $\V$, and $\V$ decomposes as $\V=\oplus_{i}\V_{\theta_i}$ with respect to the action of $Z'$, where $Z'|_{\V_{\theta_i}}=\theta_i\text{Id}_{\V_{\theta_i}}$.
Therefore, $\exp(tZ')=\oplus_i e^{t\theta_i}\text{Id}_{\V_{\theta_i}}$.
For $p\in\fh_1$, we have $[Z',p]=p=Z'p-pZ'$. This implies that for all $v \in \V_{\theta_i}$ we have
 $$
Z'\cdot (p \cdot v)=(p+pZ')\cdot v = (1+\theta_i)(p\cdot v),
 $$
which implies $p \cdot v \in  \V_{\theta_i+1}$.

 \begin{lem*}\label{expandclosedorbit}
Let $\V$ be an irreducible $\tilde G_0$-module and $0\neq v\in \V$.
Suppose that the orbit $\tilde G_0\cdot[v]\subset P\V$ is closed.
Then there exists $0\neq w \in \V_{\theta_{max}}=\ker{\fh_+}$, $\theta_{max}=\max_i\{\theta_i\}$, such that
$[w]\in \tilde G_0\cdot[v]$.
 \end{lem*}

 \begin{proof}
Decompose $v=\sum_{\theta_i}v_{\theta_i}$ into $Z'$-eigenvectors as above. Let $v_{\theta_j}$ be the non-zero component of the greatest index. If $\theta_j<\theta_{max}$ then, due to irreducibility, there exists
$p \in (\tilde\fg_0)_1$ such that $p\cdot v \neq 0$ (in the opposite case $v$ generates a proper submodule).
Then $w_0 =\exp(\tau p)v$ for small $\tau>0$ has a non-zero component in $\V_{\theta_j+1}$. Repeat this procedure
for $w_0,w_1,\ldots, w_{k-1}$ until $w_k$ has non-zero component in $\V_{\theta_{max}}$
(if $\theta_j=\theta_{max}$ then $w_k=v$). This takes a finite number of operations, because $\theta_{max}$
is finite and $\theta_i\in\mathbb{Z}$, $-\theta_{max}\leq\theta_i\leq\theta_{max}$.
Since the greatest eigenvalue dominates and the orbit of $v$ is closed, the limit
$\langle w \rangle=\lim\limits_{t\to+\infty}e^{tZ'}\langle w_k \rangle \in P\V_{\theta_{max}}$ exists.
Here $\langle w_k \rangle=\R\cdot w_k$. Moreover, there exists an element $w$ in $\V_{\theta_{max}}$ which is projected to the limit $\langle w \rangle$.
 \end{proof}

 \subsection{Minimal orbits in the curvature module}
The irreducible $G_0$-submodule $H^2_2\subset H^2(\fg_{-},\fg)$ will be denoted in this subsection
by $\V^{\rm II}$ to indicate its homogeneity $2$. Since the grading element $Z$ acts on it
by multiplication by $2$, it cannot be in the annihilator of $\ka_2$. It follows from the description of the (complexified) curvature module in the previous section that the action of $\frak{sp}(1)$ preserves $\ka_2$ and is always contained in the annihilator. Therefore we can restrict our attention to the action of
$\tilde \fg_0=\frak{sl}(n,\H)$.
With respect to it the curvature module has the highest weight $\omega_1+3\omega_{2n-1}$
(we flip the Satake diagram), and hence can be identified with an irreducible real $\tilde G_0$--module
 $$
\V^{\rm II}=S^3_\C\H^{*n} \odot \H^{n}.
 $$
Here $\odot$ denotes the Cartan product (kernel of the tensor contractions), and we use the complex notations
(the complex tensor products are taken with respect to an arbitrary invariant complex structure, say
$I\in\op{Im}\H\subset\text{End}(\H^n)$, whose choice is inessential).
For real description we refer to \cite[Proposition 4.1.8]{parabook}, see also Remark \ref{realtensorproj}, but we use the complex notations (even in describing the real objects).

We would like to find a $\tilde G_0$-orbit of minimal dimension (closed by Lemma \ref{minorbitisclosed})
in $P\V^{\rm II}$. Due to Lemma \ref{expandclosedorbit} we can assume it is represented by a non-zero element
$w\in\V^{\rm II}_{\theta_{max}}=\ker{\fh_+}$. The element $w$ has pure grading with respect to $Z'$ and
hence its annihilator in $\tilde\fg_0$ is also graded:
$\mathfrak{ann}(w)=\!\mathop{\oplus}\limits_{s=-2}^2\mathfrak{ann}(w)_s$ 
$\bigl(\,=\!\!\mathop{\oplus}\limits_{s=-1}^1\mathfrak{ann}(w)_s$ for $n=2\bigr)$.
We already know that $\mathfrak{ann}(w)_+=\fh_+$.

 \begin{lem*}\label{nonegstab-II}
We have: $\mathfrak{ann}(w)\cap(\tilde\fg_0)_{-}=0$, i.e.\ $\mathfrak{ann}(w)_{-}=0$.
 \end{lem*}

 \begin{proof}
Let us consider the case $n>2$ (The case $n=2$ is a simple adaptation). The $\fh_0$ module $(\tilde\fg_0)_{-1}$ is reducible -- it is the sum of two irreps: $(\tilde\fg_0)_{-1}'\equiv$ (the first column in the matrix from $\mathfrak{sl}(n,\H)$
with the first and last entries zero) and $(\tilde\fg_0)_{-1}''\equiv$ (the last row in the matrix from $\mathfrak{sl}(n,\H)$ with the first and last entries zero). This is also true when we restrict to $\mathfrak{sl}(n-2,\H)\subset\fh_0$.

Let $q=q'+q''\in\mathfrak{ann}(w)_{-1}$ be non-zero. Then, since
$\mathfrak{sl}(n-2,\H)\subset\mathfrak{ann}(w)$ and both $(\tilde\fg_0)_{-1}'$ and $(\tilde\fg_0)_{-1}''$
are $\mathfrak{sl}(n-2,\H)$-modules, we conclude that at least one of $(\tilde\fg_0)_{-1}'$ and $(\tilde\fg_0)_{-1}''$ is entirely in the annihilator. But then, since $\fh_1\subset\mathfrak{ann}(w)$
(and computing the brackets), we conclude that at least one of
$\mathfrak{sp}(1)_{\rm left}$ and $\mathfrak{sp}(1)_{\rm right}$
is entirely in $\mathfrak{ann}(w)$, which is impossible.

Thus $\mathfrak{ann}(w)_{-1}=0$. If there is $0\neq q\in\mathfrak{ann}(w)_{-2}$, then taking brackets
with $\fh_1\subset\mathfrak{ann}(w)$ we get a non-zero element in $\mathfrak{ann}(w)_{-1}$, which is impossible
by the above. This contradiction proves the claim.
 \end{proof}

Now due to the highest weight, $\mathfrak{sl}(n-2,\H)$ acts trivially on $\V^{\rm II}$, and so
from ${\fh_0}^{ss}$ only $\frak{sp}(1)^2=\mathfrak{sp}(1)_{\rm left}\oplus\mathfrak{sp}(1)_{\rm right}$ acts non-trivially. With respect to this algebra the module $\V^{\rm II}_{\theta_{max}}$ has highest weight
$\omega_1+3\omega_2$, and as an irreducible {\em real} module it has real dimension $8$.
We want to maximize the annihilator of an element $w$.

 \begin{lem*}
Dimension of the $\frak{sp}(1)^2$-orbit through a non-zero element $w\in\V^{\rm II}_{\theta_{max}}$ is
either $5$ or $6$. Thus nontrivial annihilator can be only $\frak{so}(2)\subset\frak{sp}(1)^2$ of $\dim=1$.
 \end{lem*}

 \begin{proof}
The complex $\frak{sp}(1)^2$-module $\V^{\rm II}_{\theta_{max}}\otimes\C$ of the highest weight $\omega_1+3\omega_2$
is the outer product of the irreducible $\sp(1)$--modules $\C^2$ and $S^3_\C\C^2$. The algebra $\sp(1)^2$ is a compact real form of rank $2$, therefore the subalgebra of dimension $2$ is a maximal torus $\mathbb{T}^2$ that is unique up to conjugation. Any subalgebra of dimension $>2$ contains such a maximal torus, but $\mathbb{T}^2$ does not preserve any vector in the module. Therefore the maximal possible annihilator dimension is $1$, and any subalgebra of dimension $1$ is isomorphic to $\mathfrak{so}(2)$. This is realized by annihilator of the highest weight vector,
and its real part has the same annihilator. This annihilator is generated by $3e_{\rm left}-e_{\rm right}$,
where $e_{\rm left}$ and $e_{\rm right}$ are generators of Cartan subalgebras in the two ideals of $\sp(1)^2$.
 \end{proof}

 \begin{cor*}\label{corr1}
The largest annihilator of a non-zero $w\in\V^{\rm II}_{\theta_{max}}$ with respect to the action of $\fg_0$
is $\frak{sp}(1)\oplus\bigl(\frak{so}(2)\oplus\R\oplus\frak{gl}(n-2,\H)\bigr)\ltimes\fh_+$,
where $\R$ is generated by a suitable linear combination of the grading elements $Z$ and $Z'$ of $\fg$
and $\tilde \fg_0$.
 \end{cor*}

We realize this annihilator in complex notations as follows. Let $\H^n=\H_1\oplus\ldots\oplus\H_n$ and
$v_m$ be the standard basis of $\H_m$ ($v=1,i,j,k$), $v^*_m$ be the real dual basis, $m=1,\dots,n$. Denote by $q_{r,s}\in\frak{sl}(n,\H)$ the matrix that contains $q$ on the $r$-th row
and $s$-th column, and that contains zeros elsewhere. The action on
$\H^n$ is $q_{r,s}\cdot v_t=(qv)_r\,\delta_{s,t}$ and the action on $\H^{*n}$ is minus the transpose.

Let $w=1_n^{*3}\otimes1_1\in\V^{\rm II}$ (this element is contained in the Cartan product
because $\langle 1_n^*,1_1 \rangle=0$, so the tensor contractions yield zero).
Then $\mathfrak{ann}(w)$ in $\tilde\fg_0$ is generated by $q_{r,s}$ for $1\le r<n$, $1<s\leq n$ ($q=1,i,j,k$),
where if $r=s$ and $q$ real we have to compensate by $1_{n,n}+3\cdot1_{1,1}$,
and the element $i_{n,n}+3\cdot i_{1,1}$. To get the annihilator in $\fg_0$ we add $\sp(1)$ and the element
$Z+1_{1,1}-1_{n,n}$.

\begin{rem*}\label{realtensorproj}
The element $w=1_n^{*3}\otimes1_1\in\V^{\rm II}$ is actually written in complex tensor notation. To get it as a real tensor, one should project the corresponding real tensor product to its complex linear submodule, and then take another projection to a self-conjugate submodule with respect to an invariant complex conjugation. We note that the first projection depends on the choice of invariant complex structure. We choose $i\in\text{Im}(\H)$. Then
\begin{align*}
\text{proj}_i (1_n^{*3}\otimes1_1) = \tfrac{1}{8}(&1_n^{\ast 3}\otimes 1_1+3\cdot1_n^{\ast 2}i_n^{\ast}\otimes i_1-3\cdot1_n^{\ast} i_n^{\ast 2}\otimes 1_1-i_n^{\ast 3}\otimes i_1).
\end{align*}
A complex conjugation can be chosen as (right) multiplication by $j$ of all tensor factors of a monomial, extended by linearity. A stable element is then given by the projector $\text{proj}_{\text{sc}}$ to the self-conjugate submodule.
\begin{align*}
\text{proj}_{\text{sc}}(\text{proj}_i (1_n^{*3}\otimes1_1)) = \tfrac{1}{16}(&1_n^{\ast 3}\otimes 1_1+3\cdot1_n^{\ast 2}i_n^{\ast}\otimes i_1-3\cdot1_n^{\ast} i_n^{\ast 2}\otimes 1_1-i_n^{\ast 3}\otimes i_1+\\
+&j_n^{\ast 3}\otimes j_1+3\cdot j_n^{\ast 2}k_n^{\ast}\otimes k_1-3\cdot j_n^{\ast} k_n^{\ast 2}\otimes j_1-k_n^{\ast 3}\otimes k_1).
\end{align*}
Note that the symmetric tensor products come with factors of $\frac{1}{3}$, which will cancel out the factors of 3 in our formula, so that e.g. the coefficient of $1_n^{\ast}\otimes 1_n^{\ast} \otimes i_n^{\ast}\otimes i_1$ is $\frac{1}{16}$. The tensor $\text{proj}_{\text{sc}}(\text{proj}_i (1_n^{*3}\otimes1_1))$ has the required annihilator, index symmetries, vanishing contraction, and so can serve as a generator of the real curvature module under the action of $\sl(n,\H)$.
\end{rem*}

Taking the semi-direct product of the annihilator in $\fg_0$ and the Abelian algebra $\H^n$, we get the graded algebra $\frak{a}^{\psi_{\rm II}}$ of maximal dimension provided $0\neq\psi_{\rm II}\in\V^{\rm II}$:
 $$
\frak{a}^{\psi_{\rm II}} = \Big(\sp(1)\oplus\bigl(\frak{so}(2)\oplus\R\oplus\mathfrak{gl}(n-2,\H)\bigr)\ltimes
\mathfrak{heis}(8n-12,\H)\Big)\ltimes\H^n.
 $$
This will be shown to be associated to the filtration on the symmetry algebra $\frak{s}$ of a geometry
with $\ka_2\neq0$ in the next section.

 \subsection{Minimal orbits in the torsion module}
The irreducible $G_0$-submodule $H^2_1\subset H^2(\fg_{-},\fg)$ will be denoted in this subsection
by $\V^{\rm I}$ to indicate its homogeneity $1$. It is a quaternionic module.
From the weighted Dynkin diagram (our Satake diagram in
Section \ref{Ssetup} with all nodes white) we see that the
the complexification $\V^{\rm I}\otimes\C$ is an outer product of the $\sp(1)$-module $S^3_\C\H$ and the $\sl(n,\H)$-module $\Lambda^2_\C\H^{*n}\odot\H^n$ (this Cartan product is the kernel of the contraction $\Lambda^2_\C\H^{*n}\otimes_\C \H^n\rightarrow \H^{*n}$). We refer to \cite[Proposition 4.1.8]{parabook}
for the description of $\V^{\rm I}$ as the real module.

One could expect that a generator of a minimal orbit can be realized as the tensor product of
such generators in each factor, and this is indeed the case.
A minimal non-zero $\sp(1)$-orbit in $S^3_\C\H$ can have dimension no less than $2$, because the maximal
proper subalgebra of $\sp(1)$ is of dimension 1. This means that any element of the torsion module which achieves maximal stabilizer in $\sl(n,\H)$ and a stabilizer of dimension 1 in $\sp(1)$ generates a minimal orbit.

Thus, we analyze the torsion module under the action of $\tilde\fg_0=\sl(n,\H)$ alone, which yields
 $$
\V^{\rm I}=\C^4\otimes_\C \Lambda^2_\C\H^{*n}\odot \H^n,
 $$
where the first factor $\C^4$ is a trivial module. This decomposes as a direct sum of modules isomorphic to $\Lambda^2_\C\H^{*n}\odot \H^n$. One can always find a minimal orbit in a completely reducible module which is contained in an irreducible summand. Moreover this orbit is closed by Lemma \ref{minorbitisclosed} and we again utilize Lemma \ref{expandclosedorbit} to ensure that the minimal orbit has an element in $\ker\fh_+$.
Using the grading element $Z'$ of $\tilde\fg_0$ we identify $\V^{\rm I}_{\theta_{max}}=\ker\fh_+$.

Since the module $\Lambda^2_\C\H^{*n}\odot \H^n$ has highest weight $\omega_1+ \omega_{2n-2}$ and is quaternionic, so $\dim_\R\V^{\rm I}_{\theta_{max}}=4$. Taking $0\neq w\in\V^{\rm I}_{\theta_{max}}$ (of pure grade),
its annihilator is a graded algebra (containing $\fh_+$).

 \begin{lem*}\label{nonegstab-I}
We have: $\mathfrak{ann}(w)\cap(\tilde\fg_0)_{-}=0$, i.e.\ $\mathfrak{ann}(w)_{-}=0$.
 \end{lem*}

 \begin{proof}
Let us consider the case $n>2$ (The case $n=2$ is a simple adaptation). First, we show that the evaluation map $(\tilde\fg_0)_{-1}\otimes\V^{\rm I}_{\theta_{max}}\to\V^{\rm I}$ is injective. Take any element $q=q'+q''\in(\tilde\fg_0)_{-1}=(\tilde\fg_0)_{-1}'\oplus(\tilde\fg_0)_{-1}''$,
where the latter splitting into irreps is the same as in the proof of Lemma \ref{nonegstab-II}.
Using the same argument as in this proof, given a non-zero annihilator element in one of these submodules,
we conclude (because $\sl(n-2,\H)$ is in the annihilator) that the whole submodule is in the annihilator.
So it is enough to check injectivity of the action on the two elements only, which are
$q'_{s,1}$ and $q''_{n,s}$ for $1<s<n$.

Notice that $\V^{\rm I}_{\theta_{max}}=\{1^*_n\wedge j^*_n\otimes v_1:v\in\H\}$.
If $q=q'_{s,1}$, then the action is $q\cdot(1^*_n\wedge j^*_n\otimes v_1)=1^*_n\wedge j^*_n\otimes(qv)_s\neq0$,
and if $q=q''_{n,s}$, then the action is
$q\cdot(1^*_n\wedge j^*_n\otimes v_1)=-(q''^*_s\wedge j^*_n+1^*_n\wedge (q''j)^*_s)\otimes v_1\neq0$.
Thus $\mathfrak{ann}(w)\cap(\tilde\fg_0)_{-1}=0$.

The rest mimics the proof of Lemma \ref{nonegstab-II}: if a non-zero annihilator element exists in
$(\tilde\fg_0)_{-2}$, then bracketing with $\fh_1$ we obtain a non-zero annihilator element in $(\tilde\fg_0)_{-1}$,
which is a contradiction.
 \end{proof}

Thus it remains to consider the action of $\fh_0$ on $\V^{\rm I}_{\theta_{max}}$.
Since in the semi-simple part $\fh^{ss}_0=\sp(1)_{\rm left}\oplus\frak{sl}(n-2,\H)\oplus\sp(1)_{\rm right}$
the last two summands are in the annihilator (because of the weight of the representation),
this reduces to considering $\V^{\rm I}_{\theta_{max}}$ as $\sp(1)=\sp(1)_{\rm left}$-module
(of the highest weight $\omega_1$). This is the standard left action of $\sp(1)$ on $\H$,
any element of $\sp(1)$ acts as a complex structure, and so this part gives no contribution to the annihilator
of any $0\neq w\in\V^{\rm I}_{\theta_{max}}$. Also, similar to the curvature module, a combination
of the grading elements acts trivially.

 \begin{cor*}\label{corr2}
The largest annihilator of a non-zero $w\in\V^{\rm I}_{\theta_{max}}$ with respect to the action of $\fg_0$ is
$\frak{so}(2)\oplus\bigl(\R\oplus\frak{gl}(n-2,\H)\oplus\frak{sp}(1)_{\rm right}\bigr)\ltimes\fh_+$,
where $\R$ is generated by a suitable linear combination of the grading elements $Z$ and $Z'$ of $\fg$ and $\tilde \fg_0$.
 \end{cor*}

Let us give the generators of this annihilator in the complex tensor notations.
Fixing $w=1^*_n\wedge j^*_n\otimes 1_1$ (again this element is contained in the Cartan product)
we conclude that $\mathfrak{ann}(w)$ in $\tilde\fg_0$ is generated by
the elements $v_{n,n}$ ($v=i,j,k$) and the elements
$q_{r,s}$ for $1\le r<n$, $1<s\leq n$ ($q=1,i,j,k$); if $r=s$ and $q$ is real, 
then $q_{r,s}$ is compensated by $1_{n,n}+2\cdot1_{1,1}$ to belong to $\tilde\fg_0$.
To get the annihilator in $\fg_0$ we add one element from $\sp(1)$ and the element $Z+1_{1,1}-1_{n,n}$.

Taking the semi-direct product of this annihilator and the Abelian algebra $\H^n$, we get the graded algebra $\frak{a}^{\psi_{\rm I}}$ of maximal dimension provided $0\neq\psi_{\rm I}\in\V^{\rm I}$:
 $$
\frak{a}^{\psi_{\rm I}} =
\Big(\mathfrak{so}(2)\oplus\bigl(\R\oplus\mathfrak{gl}(n-2,\H)\oplus\sp(1)_{\rm right}\bigr)\ltimes
\mathfrak{heis}(8n-12,\H)\Big)\ltimes\H^n.
 $$
This will be shown to be associated to the filtration on the symmetry algebra $\frak{s}$ of a geometry
with $\ka_1\neq0$ in the next section.

\begin{rem*}\label{RkFlip}
The annihilator algebras from Corollaries \ref{corr1} and \ref{corr2} are very similar but not isomorphic. The following is an explanation of this phenomenon. We reduce the curvature- and torsion-modules via the parabolic subalgebra $\p_{2,2n-2}$ (or $\p_2$ for $n=2$) of $\sl(n,\H)$. This yields the diagrams
\begin{align*}
& \addtolength{\spaceunderA}{0pt}
  \parbox[c][\spaceunderA][t]{1in}{\xymatrix@=\lengtharrowS@R=0pt{
                                                                           \\
     *\txt{$\scriptstyle \rule[\spaceundercoef]{0pt}{0mm} 3$}
    &*\txt{$\scriptstyle \rule[\spaceundercoef]{0pt}{0mm} -3$}
    &*\txt{$\scriptstyle \rule[\spaceundercoef]{0pt}{0mm} 0$}
    &*\txt{$\scriptstyle \rule[\spaceundercoef]{0pt}{0mm} 1$}
    &*\txt{$\scriptstyle \rule[\spaceundercoef]{0pt}{0mm} 0$}
    &*\txt{$\scriptstyle \rule[\spaceundercoef]{0pt}{0mm} 0$}
    &&*\txt{$\scriptstyle \rule[\spaceundercoef]{0pt}{0mm} 0$}
    &*\txt{$\scriptstyle \rule[\spaceundercoef]{0pt}{0mm} 0$}
    &*\txt{$\scriptstyle \rule[\spaceundercoef]{0pt}{0mm} 0$}
    &*\txt{$\scriptstyle \rule[\spaceundercoef]{0pt}{0mm} 1$}            \\
     *\txt{$\bullet$}       \sli[r]
    &*\txt{$\times$}       \sli[r]
    &*\txt{$\bullet$}       \sli[r]
    &*\txt{$\times$}       \sli[r]
    &*\txt{$\bullet$}       \sli[r]
    &*\txt{$\circ$}       \sli[r]
    &*\txt{$\, \cdots$}    \sli[r]
    &*\txt{$\circ$}       \sli[r]
    &*\txt{$\bullet$}       \sli[r]
    &*\txt{$\times$}       \sli[r]
    &*\txt{$\bullet$}
  }}
\\
&
\addtolength{\spaceunderA}{0pt}
  \parbox[c][\spaceunderA][t]{1in}{\xymatrix@=\lengtharrowS@R=0pt{
                                                                           \\
     *\txt{$\scriptstyle \rule[\spaceundercoef]{0pt}{0mm} 0$}
    &*\txt{$\scriptstyle \rule[\spaceundercoef]{0pt}{0mm} -4$}
    &*\txt{$\scriptstyle \rule[\spaceundercoef]{0pt}{0mm} 3$}
    &*\txt{$\scriptstyle \rule[\spaceundercoef]{0pt}{0mm} 0$}
    &*\txt{$\scriptstyle \rule[\spaceundercoef]{0pt}{0mm} 0$}
    &*\txt{$\scriptstyle \rule[\spaceundercoef]{0pt}{0mm} 0$}
    &&*\txt{$\scriptstyle \rule[\spaceundercoef]{0pt}{0mm} 0$}
    &*\txt{$\scriptstyle \rule[\spaceundercoef]{0pt}{0mm} 0$}
    &*\txt{$\scriptstyle \rule[\spaceundercoef]{0pt}{0mm} 0$}
    &*\txt{$\scriptstyle \rule[\spaceundercoef]{0pt}{0mm} 1$}            \\
     *\txt{$\bullet$}       \sli[r]
    &*\txt{$\times$}       \sli[r]
    &*\txt{$\bullet$}       \sli[r]
    &*\txt{$\times$}       \sli[r]
    &*\txt{$\bullet$}       \sli[r]
    &*\txt{$\circ$}       \sli[r]
    &*\txt{$\, \cdots$}    \sli[r]
    &*\txt{$\circ$}       \sli[r]
    &*\txt{$\bullet$}       \sli[r]
    &*\txt{$\times$}       \sli[r]
    &*\txt{$\bullet$}
  }}
\end{align*}
We note that the numbers above connected pieces correspond to the action of a semi-simple subalgebra, and the numbers above crosses only affects the scaling factors of the center of $\fg_0$. One can express this by the diagram with crosses removed
$$
\addtolength{\spaceunderA}{0pt}
  \parbox[c][\spaceunderA][t]{1in}{\xymatrix@=\lengtharrowS@R=0pt{
                                                                           \\
     *\txt{$\scriptstyle \rule[\spaceundercoef]{0pt}{0mm} 0$}
    &*\txt{$\scriptstyle \rule[\spaceundercoef]{0pt}{0mm} $}
    &*\txt{$\scriptstyle \rule[\spaceundercoef]{0pt}{0mm} 1$}
    &*\txt{$\scriptstyle \rule[\spaceundercoef]{0pt}{0mm} $}
    &*\txt{$\scriptstyle \rule[\spaceundercoef]{0pt}{0mm} 3$}
    &&*\txt{$\scriptstyle \rule[\spaceundercoef]{0pt}{0mm} 0$}
    &*\txt{$\scriptstyle \rule[\spaceundercoef]{0pt}{0mm} 0$}
    &*\txt{$\scriptstyle \rule[\spaceundercoef]{0pt}{0mm} $}
    &*\txt{$\scriptstyle \rule[\spaceundercoef]{0pt}{0mm} 0$}
    &*\txt{$\scriptstyle \rule[\spaceundercoef]{0pt}{0mm} 0$}            \\
     *\txt{$\bullet$}       \sli[]
    &*\txt{\ \ \ \ }       \sli[]
    &*\txt{$\bullet$}       \sli[]
    &*\txt{\ \ \ \ }       \sli[]
    &*\txt{$\bullet$}       \sli[]
    &*\txt{\ \ \ \ }    \sli[]
    &*\txt{$\bullet$}       \sli[r]
    &*\txt{$\circ$}       \sli[r]
    &*\txt{$\, \cdots$}     \sli[r]
    &*\txt{$\circ$}       \sli[r]
    &*\txt{$\bullet$}
  }}
$$
that is the same for the two modules, after a permutation. Hence the contribution from $\fg_0^{ss}$ to the annihilator must be abstractly isomorphic in the two cases. The difference then comes from the action of $\fg_0^{ss}$ on $\fh_+$.
\end{rem*}

 \section{Realizations of sub-maximal models}
In the previous section we found the annihilator algebras $\frak{a}^{\psi_{\rm I}}$ and
$\frak{a}^{\psi_{\rm II}}$ of maximal dimension that is $\mathfrak{U}=4n^2-4n+9$ in both cases.
To prove this is realizable, we follow the idea of \cite[\S4.2]{KT} and deform the graded bracket
structure on $\frak{a}$ to obtain a new filtered Lie algebra $\mathfrak{f}$. We use the real highest
weight vector in our modules (that correspond to the minus lowest weight vectors of the duals - note
that we used flip of the Satake diagram in our construction).

This is expected to correspond to the symmetry algebra of a submaximally symmetric model, which is
(non-flat) homogeneous with the isotropy being $\mathfrak{f}_{\ge0}=\mathfrak{a}_0$
(because of the prolongation-rigidity), and we show it is the case.

To do this we follow the approach in \cite[\S4.1]{KT} that allows to establish an abstract model,
basing on the extension functor construction. However we also provide explicit matrix models of the
corresponding almost quaternionic manifolds $(M,Q)$, for which the direct computation confirms the
amount of symmetry is submaximal $\mathfrak{S}=\mathfrak{U}$. We consider the curvature and
torsion cases separatetly. The corresponding theorems imply the main result of this paper.

We consider the cases when non-zero $\ka=(\ka_1,\ka_2)$ is either $(\psi_{\rm I},0)$ or $(0,\psi_{\rm II})$.
One could also question if the submaximal symmetry dimension can be achieved when both torsion and curvature
are non-zero, but even though abstractly the maximal annihilator
algebras $\frak{a}^{\psi_{\rm I}}\simeq\frak{a}^{\psi_{\rm II}}$ the discussion in Remark \ref{RkFlip}
shows that the answer to the above question is negative.

 \subsection{Realization of the curvature model}

Consider the case of non-vanishing curvature and vanishing torsion first. In order to realize the symmetry algebra, the minimal orbit in the abstract curvature module $\V^{\rm II}$ needs to be reinterpreted as a deformation of the graded algebra $\frak{a}^{\psi_{\rm II}}$. This can be done by finding a $G_0$-equivariant map
 $$
b: \V^{\rm II}=S^3_\C\H^{*n}\odot\H^{n} \rightarrow \mathbb{B}=\Lambda^2\H^{*n} \otimes \fg_0,
 $$
and using the Lie bracket deformation given by the image $b(v)$ of a generator
$0\neq v\in\V^{\rm II}$ of the minimal orbit $G_0\cdot[v]$.
 \begin{lem*}
The equivariant map $b$ exists and is unique up to scale.
 \end{lem*}

 \begin{proof}
The real module $\mathbb{B}$ is completely reducible, and we compute its decomposition into irreducible submodules
by finding an $\sl(n,\H)$-invariant real subspace in the complexification after applying standard methods from
the complex representation theory of $A_{2n-1}=\sl(n,\H)\otimes\C$. We have (in complexification $\H^n=\C^{2n}$):
 $$
(\Lambda^2_\R\H^{*n})^\C = \Lambda^2_\C(2\cdot\H^{*n}) = 3\cdot\Lambda^2_\C\H^{*n}\oplus S^2_\C \H^{*n}.
 $$
The Cartan product $S^2_\C\H^{*n}\odot\ad_{\sl(n,\H)}^\C$ has the same highest weight as $\V^{\rm II}$,
and so is isomorphic to it as a complex $A_{2n-1}$-module, with the isomorphism mapping the $\sl(n,\H)$-invariant real submodules into each other. This is the unique submodule in $\mathbb{B}$ of the required isomorphism type,
so the map $b$ is defined and is unique up to scalar multiplication (since $\End_{\fg_0}(\V^{\rm II})=\R$,
this scalar is a real number).
 \end{proof}

We construct $b$ in the complex tensor notations as in the previous section.
 \begin{prop*}
The bracket deformation on an extremal generator $w\in\V^{\rm II}_{\theta_{max}}$,
corresponding to a minimal orbit $G_0\cdot[w]$ in $P\V^{\rm II}$, is given by the formula:
 \begin{align*}
b(w) = &(i_n^\ast \wedge j_n^\ast-1_n^\ast \wedge k_n^\ast)\otimes(i_n^\ast \otimes 1_1-1_n^\ast \otimes j_1+j_n^\ast \otimes k_1-k_n^\ast \otimes j_1)\\
-&(1_n^\ast \wedge j_n^\ast+i_n^\ast \wedge k_n^\ast)\otimes(1_n^\ast \otimes 1_1+i_n^\ast \otimes i_1+j_n^\ast \otimes j_1+k_n^\ast \otimes k_1).
 \end{align*}
 \end{prop*}

Define the deformed Lie bracket on the space of $\mathfrak{a}^{\psi_{\rm II}}$ via $b(w)$:
 $$
[\,,\,]_{\frak{f}_{\rm II}}=[\,,\,]_{\frak{a}^{\psi_{\rm II}}}+b(w)(\,,\,).
 $$
Similarly to \cite[Lemma 4.1.1]{KT} one can check that this is a Lie bracket (the Jacobi identity holds), and
the space $\mathfrak{a}^{\psi_{\rm II}}$ equipped with it is a new (now filtered) Lie algebra $\frak{f}_{\rm II}$.

This deformation changes the Lie brackets of the previously Abelian subalgebra $\H^n$, and this subspace
$\H_n\subset \H^n$ becomes non-Abelian. The new bracket component takes values in the center of the subalgebra $\mathfrak{heis}(8n-12,\H)$:
 $$
[\H_n,\H_n]_{\frak{f}_{\rm II}}\subset \frak{z}(\mathfrak{heis}(8n-12,\H)).
 $$
Notice that the semi-simple Levi factor of the symmetry algebra is unchanged by this deformation:
$\mathfrak{f}_{\rm II}^{ss}=(\frak{a}^{\psi_{\rm II}})^{ss}=\sp(1)\oplus\sl(n-2,\H)$.
Due to the presence of the subalgebra $\sl(n-2,\H)\ltimes \H^{n-2}$ in $\mathfrak{f}_{\rm II}$,
a sub-maximal model can be realized as a direct product (although this notion is coordinate dependent) of
the flat structure $\H^{n-2}$ and a sub-maximal structure of dimension $2$.

Although the symmetry algebra is not solvable, its solvable radical acts locally transitively,
which allows us to integrate the algebra and produce a coordinate description of the model.
For $n=2$ the operator $I$ on $\H^2=\R^8(h_1,\ldots,h_8)$ is given by the matrix
 $$
I=\begin{pmatrix}
A_I & C_I \\ 0 & B_I
\end{pmatrix},\,\, A_I=\begin{pmatrix}
0 & -1 & 0 &0 \\
1 & 0 & 0 &0 \\
0 & 0 & 0 &1 \\
0 & 0 & -1 &0
\end{pmatrix}
 $$
 $$
B_I=\frac{1}{\alpha^2}\begin{pmatrix}
0 & 2h_2^2-\alpha^2 & -2\,h_2\,h_4 &2\,h_2\,h_3 \\
\alpha^2-2h_2^2 & 0 & 2\,h_2\,h_3 &2\,h_2\,h_4 \\
2\,h_2\,h_4 & -2\,h_2\,h_3 & 0 &2h_2^2-\alpha^2 \\
-2\,h_2\,h_3 & -2\,h_2\,h_4 & \alpha^2-2h_2^2 &0
\end{pmatrix}
 $$
 \begin{align*}
C_I^t=&\frac{1}{2\alpha^2}\begin{pmatrix}
0 & h_2(2\alpha^2-3h_2^2) & h_3(\alpha^2-3h_2^2) & h_4(\alpha^2-3h_2^2)\\
h_2(4\alpha^2-3h_2^2) & 0 & -h_4(h_2^2+\alpha^2) & h_3(h_2^2+\alpha^2) \\
h_4(3h_2^2-\alpha^2) & h_3(3h_2^2+\alpha^2) & h_2(3h_3^2+h_4^2) & 2h_2h_3h_4 \\
h_3(\alpha^2-3h_2^2) & h_4(3h_2^2+\alpha^2) & 2h_2h_3h_4 & h_2(h_3^2+3h_4^2)
\end{pmatrix} \\
&\!\!\!+\frac{1}{\alpha^2}\begin{pmatrix}
h_2 h_5+h_3 h_7+h_4 h_8 & -h_2 h_6-h_3 h_8+h_4 h_7 & 0 & 0\\
h_2 h_6-h_3 h_8+h_4 h_7 & h_2 h_5-h_3 h_7-h_4 h_8 & 0 & 0\\
h_2 h_7-h_3 h_5-h_4 h_6 & -h_2 h_8+h_3 h_6-h_4 h_5 & 0 & 0\\
h_2 h_8+h_3 h_6-h_4 h_5 & h_2 h_7+h_3 h_5+h_4 h_6 & 0 & 0
\end{pmatrix}
 \end{align*}
(note the transpose). Here $\alpha^2=h_2^2+h_3^2+h_4^2$. The operator $J$ is given by
 $$
J=\begin{pmatrix}
A_J & C_J \\ 0 & B_J
\end{pmatrix},\,\, A_J=\begin{pmatrix}
0 & 0 & 1 & 0\\0 & 0 & 0 & 1\\-1 & 0 & 0 & 0\\0 & -1 & 0 & 0
\end{pmatrix}
 $$
 $$
B_J=\frac{1}{\alpha^2}\begin{pmatrix}
0 & -2h_2h_3 & 2h_3h_4 & \alpha^2-2h_3^2 \\
2h_2h_3 & 0 & \alpha^2-2h_3^2 & -2h_3h_4 \\
-2h_3h_4 & 2h_3^2-\alpha^2 & 0 & -2h_2h_3 \\
2h_3^2-\alpha^2 & 2h_3h_4 & 2h_2h_3 & 0
\end{pmatrix}
 $$
 \begin{align*}
C_J^t=&\frac{1}{4\alpha^2}\begin{pmatrix}
3h_4\alpha^2 & h_3(6h_2^2-\alpha^2) & h_2(6h_3^2-\alpha^2) & 6h_2h_3h_4 \\
h_3(6h_2^2-5\alpha^2) & -3h_4\alpha^2 & 2h_2h_3h_4 & h_2(3\alpha^2-2h_3^2) \\
-6h_2h_3h_4 & 3h_2(\alpha^2-2h_3^2) & h_3(\beta^2-4h_3^2) & h_4(3\alpha^2-4h_3^2) \\
h_2(6h_3^2-\alpha^2) & -6h_2h_3h_4 & -h_4(3\alpha^2+4h_3^2) & h_3(\beta^2-4h_4^2)
\end{pmatrix}\\
 &\!\!\!+\frac{1}{\alpha^2}\begin{pmatrix}
h_2h_7-h_3h_5+h_4h_6 & 0 & h_2h_6+h_3h_8-h_4h_7 & 0 \\
-h_2h_8-h_3h_6-h_4h_5 & 0 & -h_2h_5+h_3h_7+h_4h_8 & 0 \\
-h_2h_5-h_3h_7+h_4h_8 & 0 & h_2h_8-h_3h_6+h_4h_5 & 0 \\
h_2h_6-h_3h_8-h_4h_7 & 0 & -h_2h_7-h_3h_5-h_4h_6 & 0
\end{pmatrix}
 \end{align*}
(note the transpose). Here $\beta^2=h_2^2-h_3^2-h_4^2$. Then we let $K=IJ$.

To get the quaternionic structure for general quaternionic dimension $n$,
re-denote the above operators for $n=2$ by $I_{(2)}$ and $J_{(2)}$.
Now let $I$ be given as the block matrix with $I_{(2)}$ in the top $8\times8$ block,
$A_I$ on the following diagonal $4\times4$ blocks and zeroes elsewhere.
Similarly let $J$ be given as the block matrix with $J_{(2)}$ in the top $8\times8$ block,
$A_J$ on the following diagonal $4\times4$ blocks and zeroes elsewhere.
Define $K=IJ$. Denote the obtained quaternionic structure $(I,J,K)$ by $Q_{\rm II}$.

 \begin{thm*}\label{ThmCQ}
The quaternionic structure $(M,Q_{\rm II})$ has symmetry algebra $\mathfrak{s}_{\rm II}$ of
submaximal dimension $\mathfrak{S}_2=4n^2-4n+9$.
 \end{thm*}

 \begin{proof}
The proof of \cite[Lemma 4.1.4]{KT} gives the abstract parabolic model via
the extension functor construction. The symmetry algebra $\mathfrak{s}$ of this model contains
(by construction) the deformed algebra $\mathfrak{f}_{\rm II}$ constructed above.
Thus we already have at least $\mathfrak{S}_2$ symmetries. Since this coincides with the universal upper
bound $\mathfrak{U}_2=\mathfrak{U}$, there can be no more symmetries:
$\mathfrak{s}_{\rm II}=\mathfrak{f}_{\rm II}$.
 \end{proof}


\begin{rem*}
There is a reductive decomposition $\frak{s}=\frak{h}+\frak{m}$, where $\frak{m}=\frak{g}_{-1}=\H^n$. Moreover, we have $[\frak{m},\frak{m}] \subset \frak{h}$. Thus $(\frak{h}, \frak{m})$ is a symmetric pair. This reflects the fact that the submaximally symmetric quaternionic geometry is a locally affine symmetric space in the sense of \cite{N}. Direct computations (in Maple's DifferentialGeometry package, for $n=2$) gives locally a unique quaternionic invariant connection and this connection has vanishing torsion and parallel curvature. Thus the connection corresponds to the canonical connection on the symmetric space for the pair $(\frak{h}, \frak{m})$. This local connection is hypercomplex for the above $I,J,K$, and it is the unique Obata connection \cite{Alekseevsky}. Because the structure is torsion--free, the connection is also one of the Oproiu connections and determines the class of Oproiu connections on the submaximal model. Moreover, the connection is Ricci--flat and its curvature coincides with the quaternionic Weyl curvature (and in particular it is harmonic). However, Oproiu connections in the class are not Ricci-flat in general. Finally, direct computation shows that there is no invariant metric 
(of any signature) on the submaximal model.
\end{rem*}

 \subsection{Realization of the torsion model}
The case of non-vanishing torsion and vanishing curvature can be treated similarly.
In this case, we immediately interpret the element
 $$
w=1_n^*\wedge j^*_n\otimes 1_1
 $$
as a deformation to the graded algebra $\frak{a}^{\psi_{\rm I}}$, so the deformed Lie bracket on the space of
$\mathfrak{a}^{\psi_{\rm I}}$ is
 $$
[\,,\,]_{\frak{f}_{\rm I}}=[\,,\,]_{\frak{a}^{\psi_{\rm I}}}+w(\,,\,).
 $$
As in the curvature case, the previously Abelian subalgebra $\H^n$ will become non-Abelian, but in this case we have
 $$
[\H_1,\H_1]_{\frak{f}_{\rm I}}\subset\H_n\subset \H^n,
 $$
which means that $\H^n$ remains a subalgebra. It is a nilpotent ideal of the algebra $\frak{f}_{\rm I}$
acting locally transitively on the corresponding (local) homogeneous model $F/K$ (cf.\ \cite[Lemma 4.1.4]{KT}).
Therefore, the minimal model is locally equivalent to a left invariant structure on the nilpotent Lie group
corresponding to $(\H^n,[\,,\,]_{\frak{f}_{\rm I}})$.

As in the curvature case, the semi-simple Levi factor of the symmetry algebra is unchanged by the deformation, and due to the presence of $\sl(n-2,\H)$ the model can once again be realized as a direct product of a submaximal structure in quaternionic dimension $2$ and a flat structure in quaternionic dimension $n-2$.

The matrices for $n=2$ in the torsion case turn out to be considerably simpler than in the curvature case.
Namely, the operators $I$ and $J$ on $\H^2=\R^8(h_1,\ldots,h_8)$ are given by the following matrices
and $IJ=K$.
 $$
I=\begin{pmatrix}
0 & -1 &  0 & 0 & 0   &  h_7 & 0  & 0 \\
1 &  0 &  0 & 0 & h_7 &  0   & 0  & 0 \\
0 &  0 &  0 & 1 & 0   &  0   & 0  & 0 \\
0 &  0 & -1 & 0 & 0   &  0   & 0  & 0 \\
0 &  0 &  0 & 0 & 0   & -1   & 0  & 0 \\
0 &  0 &  0 & 0 & 1   & 0    & 0  & 0 \\
0 &  0 & 0  & 0 & 0   & 0    & 0  & 1 \\
0 &  0 & 0  & 0 & 0   & 0    & -1 & 0
\end{pmatrix}
 $$
 $$
J=\begin{pmatrix}
0  &  0 & 1 & 0 & 0    & 0  & -h_7 & 0 \\
0  &  0 & 0 & 1 & 0    & 0  & 0    & 0 \\
-1 &  0 & 0 & 0 & -h_7 & 0  & 0    & 0 \\
0  & -1 & 0 & 0 & 0    & 0  & 0    & 0 \\
0  &  0 & 0 & 0 & 0    & 0  & 1    & 0 \\
0  &  0 & 0 & 0 & 0    & 0  & 0    & 1 \\
0  &  0 & 0 & 0 & -1   & 0  & 0    & 0 \\
0  &  0 & 0 & 0 & 0    & -1 & 0    & 0
\end{pmatrix}
 $$

To get the quaternionic structure for general quaternionic dimension $n$,
we again extend the above $I$ by the $4\times4$ block-matrices that form the block-diagonal
of the $8\times8$ matrix, and do similarly for $J$; then we define $K=IJ$.
Denote the obtained quaternionic structure $(I,J,K)$ by $Q_{\rm I}$.

 \begin{thm*}
The quaternionic structure $(M,Q_{\rm I})$ has symmetry algebra $\mathfrak{s}_{\rm I}$ of
submaximal dimension $\mathfrak{S}_1=4n^2-4n+9$.
 \end{thm*}

 \begin{proof}
The proof mimics that of Theorem \ref{ThmCQ}, and we conclude $\mathfrak{S}_1=\mathfrak{U}_1=\mathfrak{U}$,
and therefore $\mathfrak{s}_{\rm I}=\mathfrak{f}_{\rm I}$.
 \end{proof}

 \begin{rem*}
The submaximally symmetric almost quaternionic geometry of torsion type is locally representable as a group. Such structures always have at least a one-parameter family of invariant connections \cite{N}. Direct computations (in Maple's DifferentialGeometry package) shows that (for $n=2$) there is a six-parameter family of invariant connections, each with vanishing curvature and parallel torsion. However, only a two-parametric sub-family is quaternionic, and all the invariant quaternionic connections are hypercomplex. Among all these connections, we can find exactly one connection such that its torsion coincides with the structure torsion of the hypercomplex structure, and this is the Obata connection of the hypercomplex structure \cite{Alekseevsky}. However, the structure torsion of the hypercomplex structure differs from that of the almost quaternionic structure, because the submaximal model has non-vanishing (harmonic) torsion. There is no invariant quaternionic connection such that its torsion coincides with the structure torsion of the almost quaternionic structure. Thus, no Oproiu connection is invariant. Clearly, the class of Oproiu connections is invariant, but unlike the quaternionic submaximal model (torsion--free with curvature), there is no fixed point in the class.
 \end{rem*}

\smallskip

\begin{rem*}
Suppose the almost quaternionic structure $Q$ is induced by an almost hypercomplex structure $I,J,K$. Then the hypercomplex symmetry algebra consists of the quaternionic symmetries that preserve each of $I,J,K$ by itself. In particular, the almost hypercomplex structure inducing the sub-maximal
quaternionic structure $Q_{\rm II}$ has hypercomplex symmetry algebras of dimension $4n^2-4n+6$,
while that for the almost quaternionic structure $Q_{\rm I}$ has that dimension $4n^2-4n+8$.
\end{rem*}

%
%
%

\end{document}